\documentclass[a4paper,12pt]{article}
\usepackage[utf8]{inputenc}
\usepackage[T1]{fontenc}
\usepackage{latexsym}
\usepackage{amsmath}
\usepackage{amsfonts}
\usepackage{amsthm}
\usepackage{amssymb}
\newtheorem{Unterkapitel}{}
\newtheorem{theorem}{Theorem}

\newtheorem{claim}{Claim}

\theoremstyle{definition}
\newtheorem{definition}{Definition}

\def\P{\mathbb P}

\def\T{\mathbb T}

\def\<{\langle}
\def\>{\rangle}

\title{The problem of Horn and Tarski}
\author{Egbert Th{\" u}mmel}
\begin{document}

\maketitle

\begin{abstract}
In 1948 A. Horn and A. Tarski asked whether the notions of a $\sigma$-finite cc and a $\sigma$-bounded cc ordering are equivalent. We give a negative answer to this question.
\end{abstract}

When analyzing Boolean algebras carrying a measure, Horn and Tarski \cite{HT:48} defined the following two notions:
\begin{definition}\label{fragmentation}
An ordering $\P$ is called 
\begin{itemize}
\item [(i)] \emph{$\sigma$-bounded cc}\index{fragmentation!
$\sigma$-bounded cc} if  $\P=
{\displaystyle \bigcup_{n{\in}\omega}} P_n$, where each $P_n$ has the
$n+2$-cc.
\item [(ii)] \emph{$\sigma$-finite cc}\index{fragmentation!
$\sigma$-finite cc} if  $\P=
{\displaystyle \bigcup_{n{\in}\omega}} P_n$, where each $P_n$ has the
$\omega$-cc.
\end{itemize}
Here an ordering or its subset has the $\kappa$-cc ($\kappa$-chain condition) for a cardinal $\kappa$ if it contains no antichain (set of pairwise orthogonal elements) of size $\kappa$.

\end{definition}
Clearly, any $\sigma$-bounded cc ordering is $\sigma$-finite cc (and both are $\omega_1$-cc - also called ccc).
Horn and Tarski asked whether these two classes coincide:

\smallskip\smallskip\smallskip

\noindent{\bf Problem: A. Horn and A. Tarski 1948} \cite{HT:48}
{\it Is every $\sigma$-finite cc ordering also $\sigma$-bounded cc?}

\smallskip\smallskip\smallskip

There is a standard way how to map an ordering densely into a complete Boolean algebra. This mapping preserves our two properties. The problem of Horn and Tarski  can therefore be formulated in terms of Boolean algebras as well. 
It is easy to see that a Boolean algebra carrying a strictly positive measure is $\sigma$-bounded cc (take as $P_n$ the set of elements of measure at least $1/n$). 
If the Boolean algebra carries only a strictly positive exhaustive submeasure this property could get lost, but still the Boolean algebra will be $\sigma$-finite cc (take the same $P_n$). 
The question, whether any Boolean algebra carrying a strictly positive exhaustive submeasure carries also a strictly positive measure, is one formulation of the famous Control measure problem. 
It was therefore expected that an anticipated negative solution of this problem will give also a counterexample to the problem of Horn and Tarski . 
But when such an example solving the Control measure problem was constructed by M. Talagrand (\cite{Tal:2008}) it turned out that it is even $\sigma$-bounded cc, so the problem of Horn and Tarski  remained open. 
We will construct here a counterexample to the problem of Horn and Tarski .

\begin{theorem}
There exists an ordering which is $\sigma$-finite cc but not $\sigma$-bounded cc.
\end{theorem}

The technique used in the construction appeared first in \cite{Todo:91} and is further developed in \cite{BPT:2011} : 

For a subset $F$ of a topological space $X$ let $F^d$ denote the set of all accumulation points.
\begin{definition}
For a topological space $X$ the Todorcevic ordering $\T(X)$ is the set of all subsets $F$ of the space which are a finite union of converging sequences including their limit points.
The order relation is defined by such extensions which preserve isolated points, i.e. 
$F_1\leq F_2$ if $F_1\supseteq F_2$  and $F_1^d\cap F_2=F_2^d$.
\end{definition}

We start with the set $T=\bigcup_{\alpha<\omega_1}{ }^{\alpha+1}\omega$. This set is made into a tree by the order of inclusion. 
We will extend the order of the tree into a linear one. 
Define the order $\leq$ on $T$ by $s<t$ if either $s\subset t$ or $s(\beta)>t(\beta)$ for $\beta=\min \{\gamma\ :\ s(\gamma)\neq t(\gamma)\}$.
Note that, for any point of the tree $T$, the set of immediate successors in the tree is of type $\omega^*$ in the linear order $<$, i.e. $\omega$ ordered in the reverse. 
Take the interval topology $\tau_{\leq}$ on $T$. 
We apply the operator $\T$ on this linearly ordered topological space $(T,\tau_\leq)$ to obtain the Todorcevic ordering $\P=\T(T)$. This will be the example which proves the theorem:

\begin{claim}
$\P$ is not $\sigma$-bounded cc.
\end{claim}
\begin{proof}
Assume by contradiction that  $\P=\bigcup_{n\in\omega} P_n$, with each $P_n$ being $n+2$-cc, witnesses that $\P$ is $\sigma$-bounded cc.
For $n<\omega$ define functions $f_n:T\longrightarrow n+2$ such that $f_n(s)$ is the maximal length of an antichain which is a subset of the set $P_n(s)=\{F\in P_n : \exists t\in F^d(t\supseteq s)\}$. 
The function $f_n$ is decreasing with respect to $\subseteq$. It follows that for any $s\in T$ there is an $s^\prime\supseteq s$ such that $f_n(s^\prime)=f_n(t)$ for all $t\supseteq s^\prime$.
We find an increasing (with respect to $\subseteq$) sequence $\{s_n\}$ such that $f_n(s_n)=f_n(t)$ for all $t\supseteq s_n$.
For an arbitrary $s\in T$ with $s\supset\bigcup_{n<\omega}s_n$ we have therefore $f_n(s)=f_n(t)$ for all $t\supseteq s$ and $n<\omega$.
Fix such an $s$ and let $f(n)=f_n(s)$.
For $n<\omega$ choose in $P_n(s^\smallfrown n)$ an antichain $\{F_{n,i}\}_{i<f(n)}$ and $t_{n,i}\supseteq s^\smallfrown n$ such that $t_{n,i}\in(F_{n,i})^d$ for $i<f(n)$.
Then $\{t_{n,i}\}_{n<\omega,i<f(n)}$ converges to $s$ (if not finite) and so does $\{s^\smallfrown n\}_{n<\omega}$, i.e. $F=\{t_{n,i}\}_{n<\omega,i<f(n)}\cup\{s^\smallfrown n\}_{n<\omega}\cup\{s\}\in\P$. 
Notice that $F$ is orthogonal to all $F_{n,i}$ for $n<\omega$ and $i<f(n)$ since $t_{n,i}$ is isolated in $F$ and an accumulation point in $F_{n,i}$. 
But $F$ has to be contained in some $P_n$, hence $\{F_{n,i}\}_{i<f(n)}\cup\{F\}$ is an antichain in $P_n(s)$ and therefore $f_n(s)\geq f(n)+1$, a contradiction.
\end{proof}

\begin{claim}
$\P$ is $\sigma$-finite cc.
\end{claim}
\begin{proof}
We argue in the order $\leq$.
The set $\{s^\smallfrown k\}_{k<\omega}$ is a decreasing sequence with infimum $s$.
We can therefore for any $F\in\P$ fix a $k(F)<\omega$ such that, for $s\in F^d$, the open intervals $(s,s^\smallfrown k(F))$ are disjoint from $F^d$. No increasing sequence of $(T,\leq)$ has a supremum. 
This means that any sequence which converges to $s$ is above $s$ with the possible exception of finitely many elements.
Therefore $R(F)=F\backslash\big(\bigcup_{s\in F^d}(s,s^\smallfrown k(F))\cup F^d\big)$ is finite.  
Let $$P_{k,n,m}=\{F\in\P\quad :\quad k(F)=k\ \&\ |F^d|=n\ \&\ |R(F)|=m\}.$$
Surely $\P=\bigcup_{k,n,m<\omega}P_{k,n,m}$.
We have to show that all $P_{k,n,m}$'s are finite cc.
Assume by contradiction that $\{F_i\}_{i<\omega}\subset P_{\bar k,\bar n,\bar m}$ is an infinite antichain for some fixed $\bar k,\bar n,\bar m$.
Let $(F_i)^d=\{s_i^n\}_{n<\bar n}$ and $R(F_i)=\{r^m_i\}_{m<\bar m}$ be enumerated and put $F^n_i=F_i\cap (s^n_i,s^n_i{ }^\smallfrown\bar k)$.
Then $F^n_i$ is a sequence with limit $s^n_i$ and $F_i\backslash (F_i)^d=\bigcup_{n<\bar n}F^n_i\cup\{r^m_i\}_{m<\bar m}$ is the set of isolated points of $F_i$.
We say that $\{i,j\}\in[\omega]^2$, $i<j$, has color
 
\begin{eqnarray*}
 &(1,n,n^\prime)&\text{if } s^{n}_i\in F^{n^\prime}_j\\
  &(2,n,m)&\text{if } s^n_i=r^m_j\\
 &(3,n,n^\prime)&\text{if } s^{n}_j\in F^{n^\prime}_i\\
 &(4,n,m)&\text{if } s^n_j=r^m_i
\end{eqnarray*}
for $n,n^\prime<\bar n$ and $m<\bar m$.
Since $\{F_i\}_{i<\omega}$ was assumed to be an antichain, there must be for any $\{i,j\}\in[\omega]^2$ a point which is isolated in $F_i$ and not isolated in $F_j$ or vice versa, i.e. any pair $\{i,j\}$ obtains at least one color. 
Ramsey's theorem asserts that there must be an infinite subset of $\omega$ which is homogeneous in one color.
For notational convenience, we assume that $\omega$ itself is this homogeneous set. We are going to derive a contradiction for each of the colors.

\begin{Unterkapitel}
$\omega$ is homogeneous in color $(1,n,n^\prime)$.
\end{Unterkapitel}

Note that $s\in (t,t^\smallfrown\bar k)$ implies $s\supset t$ and $(s,s^\smallfrown\bar k)\subset (t,t^\smallfrown\bar k)$.

Homogeneity in color $(1,n,n^\prime)$ implies $s^{n}_i\in F^{n^\prime}_j\subseteq (s^{n^\prime}_j,s^{n^\prime}_j{ }^\smallfrown\bar k)$, i.e. $s^{n}_i\supset s^{n^\prime}_j$ for all $i<j$.
We have $s^{n}_{i-1}\supset s^{n^\prime}_i, s^{n^\prime}_{i+1}$, hence $s^{n^\prime}_i\subseteq s^{n^\prime}_{i+1}$ or $s^{n^\prime}_i\supset s^{n^\prime}_{i+1}$.
Consider the first case. The order $\leq$ is stronger than $\subseteq$, therefore $s^{n^\prime}_i\leq s^{n^\prime}_{i+1}<s^n_{i-1}\in F^{n^\prime}_i \subseteq (s^{n^\prime}_i,s^{n^\prime}_i{}^\smallfrown\bar k)$. 
The latter is an interval, hence $s^{n^\prime}_{i+1}=s^{n^\prime}_i$ or $s^{n^\prime}_{i+1}\in (s^{n^\prime}_i,s^{n^\prime}_i{}^\smallfrown\bar k)$, 
therefore $(s^{n^\prime}_{i+1},s^{n^\prime}_{i+1}{ }^\smallfrown\bar k)\subseteq (s^{n^\prime}_i,s^{n^\prime}_i{ }^\smallfrown\bar k)$.
But $s^n_i\notin (s^{n^\prime}_i,s^{n^\prime}_i{ }^\smallfrown\bar k)$. 
This follows from the definition of $\bar k=k(F_i)$ at the beginning of the proof.
On the other hand, $s^n_i\in F^{n^\prime}_{i+1}\subseteq (s^{n^\prime}_{i+1},s^{n^\prime}_{i+1}{ }^\smallfrown\bar k)$ by homogeneity - a contradiction.
So the second case $s^{n^\prime}_i\supset s^{n^\prime}_{i+1}$ must hold for all $i<\omega$, i.e. the $s^{n^\prime}_i$'s are a strictly decreasing sequence in the tree $T$, again a contradiction.

\begin{Unterkapitel}
$\omega$ is homogeneous in color $(2,n,m)$.
\end{Unterkapitel}
From $s^n_1=r^m_2$ and $s^n_0=r^m_2$ and $s^n_0=r^m_1$ (homogeneity in color $(2,n,m)$)
we conclude $s^n_1=r^m_1$ - a contradiction since $s^n_1$ is an accumulation point in $F_1$ and $r^m_1$ is isolated in $F_1$.
 
\begin{Unterkapitel}
$\omega$ is homogeneous in color $(3,n,n^\prime)$.
\end{Unterkapitel}
Assume that there are $i<j$ such that $s^{n}_i=s^{n}_j$.
Then $s^{n}_i=s^{n}_j\in F^{n^\prime}_i$, but $s^{n}_i$ is an accumulation point of $F_i$ whereas $F^{n^\prime}_i$ contains only isolated points of $F_i$ - a contradiction.
So the $s^{n}_j$'s are pairwise different for $j<\omega$. 
Homogeneity in color $(3,n,n^\prime)$ implies that all $s^{n}_j$, $j>0$, are in $F^{n^\prime}_0$, the set $\{s^{n}_j\}_{j=1}^\omega$ therefore converges to $s^{n^\prime}_0$.
By the same argument, we obtain that $\{s^{n}_j\}_{j=2}^\omega$ converges to $s^{n^\prime}_1$, hence $s^{n^\prime}_0=s^{n^\prime}_1$. 
Again by homogeneity $s^{n}_1\in F^{n^\prime}_0\subseteq (s^{n^\prime}_0,s^{n^\prime}_0{ }^\smallfrown\bar k)=(s^{n^\prime}_1,s^{n^\prime}_1{ }^\smallfrown\bar k)$ - a contradiction since 
$s^{n}_1\notin (s^{n^\prime}_1,s^{n^\prime}_1{ }^\smallfrown\bar k)$ by definition of $\bar k=k(F_1)$.

\begin{Unterkapitel}
$\omega$ is homogeneous in color $(4,n,m)$.
\end{Unterkapitel}
The same as color $(2,n,m)$.

\smallskip

\noindent For all the colors we obtained a contradiction, so an infinite antichain cannot exist.

\end{proof}

\def\cprime{$'$}

%\bibliographystyle{alpha}
%\bibliography{bib}
\end{document}